\documentclass{amsart}
\usepackage[utf8]{inputenc}

\usepackage{hyperref}
\usepackage[dvipsnames]{xcolor}
\usepackage{enumitem}
\usepackage{tikz-cd}
\tikzset{
	symbol/.style={
		draw=none,
		every to/.append style={
			edge node={node [sloped, allow upside down, auto=false]{$#1$}}}
	}
}
\usepackage{mathtools,amsmath,amssymb}

\usepackage{ifthen}

\usepackage{tikz}
\usetikzlibrary{cd,fadings,arrows}
\usetikzlibrary{decorations.markings}

\theoremstyle{plain}
\newtheorem{theorem}[subsubsection]{Theorem}
\newtheorem{lemma}[subsubsection]{Lemma}
\newtheorem{prop}[subsubsection]{Proposition}
\newtheorem{cor}[subsubsection]{Corollary}

\theoremstyle{definition}
\newtheorem{defn}[subsubsection]{Definition}

\newtheorem{remark}[subsubsection]{Remark}

\def\AA{\mathbb{A}}

\def\CC{\mathbb{C}}

\def\EE{\mathbb{E}}

\def\GG{\mathbb{G}}

\def\PP{\mathbb{P}}

\def\RR{\mathbb{R}}

\def\TT{\mathbb{T}}

\def\ZZ{\mathbb{Z}}

\newcommand\cD{\mathcal{D}}
\newcommand\cE{\mathcal{E}}

\newcommand\cH{\mathcal{H}}

\newcommand\cM{\mathcal{M}}

\newcommand\cO{\mathcal{O}}

\newcommand\cR{\mathcal{R}}
\newcommand\cS{\mathcal{S}}

\newcommand\cV{\mathcal{V}}

\def\bQ{\mathbf{Q}}

\newcommand\frF{\mathfrak{F}}

\newcommand\frg{\mathfrak{g}}

\newcommand{\diff}{\textup{diff}}

\newcommand\Perv{\textup{Perv}}

\renewcommand{\Re}{\textup{Re}}

\newcommand\GL{\textup{GL}}

\newcommand\SL{\textup{SL}}

\newcommand\SO{\textup{SO}}

\newcommand\Sp{\textup{Sp}}

\newcommand{\Gm}{\GG_m}
\newcommand{\GM}[1]{\GG_{m,#1}}

\newcommand{\iso}{\simeq}
\newcommand{\conv}{\overset{+}{\otimes}}

\newcommand\quash[1]{}

\renewcommand\a\alpha
\renewcommand\b\beta
\newcommand\g\gamma
\newcommand\G\Gamma
\renewcommand\d\delta

\newcommand{\x}{\chi}

\newcommand{\fa}{\mathfrak{a}}
\newcommand{\fb}{\mathfrak{b}}

\title{Stokes matrices for Airy equations}
\author{Andreas Hohl}
\author{Konstantin Jakob}
\thanks{K.J. is supported by the DFG Research Fellowship JA 2967/1-1.}
\address{(A.H.) Université Paris Cité and Sorbonne Université, CNRS, IMJ-PRG, 75013 Paris, France}
\email{andreas.hohl@imj-prg.fr}
\address{(K.J.) Fachbereich Mathematik, TU Darmstadt, Schlossgartenstraße 7, 64289 Darmstadt, Germany}
\email{konstantin.jakob@outlook.de}


\begin{document}

\maketitle

\begin{abstract}
    We compute Stokes matrices for generalised Airy equations and prove that they are regular unipotent (up to multiplication with the formal monodromy). This class of differential equations was defined by Katz and includes the classical Airy equation. In addition, it includes differential equations which are not rigid. Our approach is based on the topological computation of Stokes matrices of the enhanced Fourier--Sato transform of a perverse sheaf due to D'Agnolo, Hien, Morando and Sabbah.
\end{abstract}

\section{Introduction}

The Stokes phenomenon was first described by G.\ G.\ Stokes in \cite{Sto} in the example of the Airy differential equation in the complex plane
$$y''(x) = xy(x).$$
This equation has an irregular singular point at $x=\infty$, and Stokes showed how a formal (non-convergent) solution at the singularity determines an analytic (convergent) solution on any sufficiently small sector centred at $\infty$. However, such an analytic solution depends on the direction, and discontinuities in the following sense may appear: If we fix a formal solution, compute an analytic lift in a certain direction and continue it analytically along a path to a different direction, this analytic continuation might not be an analytic lift of the same formal solution. This observation is nowadays referred to as the \emph{Stokes phenomenon}, and it has turned out to be a key observation in the classification of irregular singularities of differential equations. Different descriptions of this phenomenon have been developed, among which the concept of Stokes matrices is a very natural approach: Fixing a basis of formal solutions, one covers a neighbourhood of the singularity by sectors, each of which admits a basis of solutions consisting of analytic lifts, and one describes the relations between these bases on the overlaps of the sectors by matrices. However, it is often difficult to determine these matrices explicitly.

A generalised Airy equation of type $(n,m)$ is an equation corresponding to a differential operator of the form
\[P_n(\partial) + Q_m(x) \] 
on $\AA^1$, where $P$ and $Q$ are complex polynomials of degrees $n$ and $m$ respectively. The differential equation is of rank $n$ and it is regular outside $\infty$. It has a single slope $(n+m)/n$ at $\infty$. In this article, we will focus on the operator $\partial^n+x^m$ with $n$ and $m$, coprime.\par
For $m=1$ the equation defined by this operator is always rigid as it is obtained by Fourier--Laplace transform from an exponential equation of rank $1$. Rigid differential equations are of interest in several areas of algebraic geometry and number theory, for example in explicit approaches to the geometric Langlands program, see \cite{Yun17} for a survey and further references. In particular it is of interest to compute differential Galois groups of the equations in question. By a theorem of Ramis \cite[Theorem 8.10]{vdP03} the differential Galois group of an irregular equation is generated by its formal monodromy, its exponential torus and the Stokes matrices. The Stokes matrices can therefore be used to determine the differential Galois groups. \par
In contrast, when $n,m>1$, the corresponding differential equation is not rigid. A lot of explicit computations of Stokes matrices for equations which do not directly arise from a regular singular equation by a single Fourier--Laplace transform in the literature concern rigid equations, e.g.\ the classical Airy equation or confluent hypergeometric equations which are studied for example in \cite{DuM89} and \cite{H20}. \par
In this paper, we give a recipe to determine the Stokes matrices explicitly for any pair of coprime integers $(n,m)$ and prove that the unipotent Stokes matrices are regular in the sense that they have minimal centraliser dimension. In particular, in the rigid case ($m = 1$) there is a closed formula for the Stokes matrices. In the end we recover results on the differential Galois groups of generalised Airy equations due to N.\ Katz \cite[§4.2]{Ka87} and M.\ Kamgarpour and D.\ Sage \cite[§4.5]{KS20}. The latter generalised the Airy equation to a $G$-connection for an arbitrary simple complex group $G$. \par
Note that Fourier--Laplace transform exchanges Airy equations of type $(n,m)$ with equations of type $(m,n)$. We compute Stokes matrices regardless of whether $m<n$ or $n<m$. Our results may therefore serve to compare Stokes matrices before and after Fourier--Laplace transform. This may shed light on the behaviour of Stokes data under Fourier--Laplace transform in the case that none of the involved equations is regular singular. Questions of this type have been studied in several works (see \cite{Moc20} and the references therein).

Our approach is based on results about the Stokes data attached to irregular singularities that arise as Fourier--Laplace transforms of regular systems. They were first presented by B.\ Malgrange in \cite[Chap.\ XII]{Mal91} and the Stokes matrices have been made more explicit by P.\ Boalch in \cite[§2]{Bo15}. In \cite{DHMS20}, A.\ D'Agnolo, M.\ Hien, G.\ Morando and C.\ Sabbah recover these formulae making use of topological results: A regular holonomic D-module $\cM$ on the affine line corresponds to a perverse sheaf via the classical Riemann--Hilbert correspondence, and this perverse sheaf can be described by means of a quiver. On the other hand, the Fourier--Laplace transform $\frF(\cM)$ of $\cM$ has an irregular singularity at $\infty$, hence one can associate Stokes matrices to this system. Via the Riemann--Hilbert correspondence of D'Agnolo--Kashiwara (see \cite{DK16}), $\frF(\cM)$ corresponds to an enhanced ind-sheaf, which must therefore encode the Stokes matrices. The authors used these topological methods in order to give an explicit description of the Stokes matrices at $\infty$ of $\frF(\cM)$ in terms of the quiver associated to the perverse sheaf describing $\cM$. This approach yields a powerful method for explicit computations, since by these topological considerations -- despite their abstractness at first sight -- one can avoid analytic considerations which are often much more complicated.

In the last section of their article, they give examples of how this result can also be used to compute Stokes matrices for certain equations which are not Fourier--Laplace transforms of regular systems -- for example the classical Airy equation -- by manipulating them in such a way that the calculation reduces to such a Fourier--Laplace transform.

In his work \cite{H20}, M.\ Hien used the results of \cite{DHMS20} to determine the Stokes multipliers of hypergeometric systems with an irregular singularity at $\infty$. By a result of Katz, the latter can be expressed (up to ramification) as Fourier--Laplace transforms of hypergeometric systems with only regular singularities, and consequently the main difficulty is to describe the quiver associated to such a regular system.

\subsection*{Outline} It is the aim of this article to give explicit computations of Stokes matrices for operators of the form $\partial^n-x^m$ for $n,m$ coprime. 
 
After briefly reviewing the main ingredients concerning perverse sheaves, quivers, Fourier--Laplace transforms and enhanced ind-sheaves from \cite{DHMS20} in Section~\ref{sect:pre}, we show how we can relate such a differential system to a regular singular hypergeometric equation (irreducible if $n$ and $m$ are coprime) in Section~\ref{sect:general}. The idea of the computation in §\ref{subsec:AiryHyp} comes from the computation of \cite{DHMS20} in the case of the classical Airy equation: It aims at expressing the de-ramified Airy module as a pullback of some Fourier transform of a system with regular singularities, and the latter turns out to be hypergeometric. We will describe its quiver with methods similar to \cite{H20}. The main theorem is then Theorem~\ref{thm:stokesreg}, computing the Stokes matrices at $\infty$ associated to an operator $\partial^n-x^m$ and proving that they are regular unipotent (up to multiplication with the formal monodromy). Even though we do not give an explicit closed formula in full generality, the proof of Theorem~\ref{thm:stokesreg} gives a recipe for computing the Stokes matrices for any given $n$ and $m$. In fact, as mentioned before, for $m=1$ we do arrive at a closed formula. Finally, using an argument of E.\ Frenkel and D.\ Gross from \cite{FG09}, we compute the differential Galois groups of the equations in question, using most prominently the fact that the unipotent Stokes matrices are regular.

\subsection*{Notation \& conventions} We will denote by $\AA^1$ the complex affine line and set $\Gm\vcentcolon=\AA^1\setminus \{0\}$. We will often emphasise the name of the affine coordinate at $0$ by an index, e.g.\ $\AA^1_z$, $\GM{z}$ are the above spaces with affine coordinate $z$. They will mostly be considered as algebraic varieties. However, when we talk about fundamental groups, we will consider the spaces equipped with their analytic topologies.

\subsection*{Acknowledgements}
We came in contact with generalised Airy equations through M.\ Kamgarpour and D.\ Sage's work on Coxeter connections and we wish to thank them for that. We thank Claude Sabbah and Marco Hien for discussions and comments that helped improve a preliminary version of this article. 

\section{Preliminaries}\label{sect:pre}

\subsection{Enhanced Riemann--Hilbert and Fourier--Laplace}\label{subsec:enhanced}
If $X$ is a smooth complex algebraic variety, we will denote by $\cD_X$ the sheaf of differential operators on $X$, by $\mathrm{D}^\mathrm{b}(\cD_X)$ the derived category of $\cD_X$-modules and by $\mathrm{D}^\mathrm{b}_{\mathrm{hol}}(\cD_X)$ the full subcategory consisting of complexes with holonomic cohomologies. Moreover, we write $\otimes^\mathrm{D}$ for the tensor product of $\cD_X$-modules and, if $f$ is a morphism of smooth complex algebraic varieties, we will denote the (derived) direct and inverse image operations for D-modules by $f_*$, $f_!$ and $f^*$, respectively, and the middle extension along an embedding $j$ by $j_{!*}$. We refer to \cite{HTT}, \cite{Kas} for details on the theory of D-modules.

We denote by $\Perv_\Sigma(\CC_X)$ the category of perverse sheaves on $X$ with singularities at the finite set $\Sigma$. This is a subcategory of $\mathrm{D}^\mathrm{b}(\CC_{X^\mathrm{an}})$, the derived category of sheaves of $\CC$-vector spaces on the complex manifold $X^\mathrm{an}$ associated to $X$. We will still write $f^*$ for the inverse image operation on $\Perv_\Sigma(\CC_X)$ and $\mathrm{D}^\mathrm{b}(\CC_{X^\mathrm{an}})$, without emphasizing the analytification involved. We will also denote the middle extension for perverse sheaves along an embedding $j$ by $j_{!*}$.

In \cite{DK16}, the authors proved a Riemann--Hilbert correspondence for (possibly irregular) holonomic D-modules on complex manifolds. Concretely, they established the fully faithful functor of \emph{enhanced solutions}
\begin{equation}\label{eq:SolEanalytic}
\cS ol^\mathrm{E}_X\colon \mathrm{D}^\mathrm{b}_\mathrm{hol}(\cD_X)^\mathrm{op}\hookrightarrow \mathrm{E}^\mathrm{b}(\mathrm{I}\CC_X)
\end{equation}
on a complex manifold $X$ (where $\cD_X$ is the sheaf of analytic differential operators on $X$).
Here, $\mathrm{E}^\mathrm{b}(\mathrm{I}\CC_X)$ denotes the triangulated category of \emph{enhanced ind-sheaves} on $X$, which has been constructed in \cite{DK16}, building on the theory of ind-sheaves \cite{KS01}.

If $X$ is a smooth complex algebraic variety, it is well-known that an (algebraic) $\cD_X$-module $\cM$ can be considered as an (analytic) $\cD_{\widehat{X}^\mathrm{an}}$-module for a smooth completion $j\colon X\hookrightarrow \widehat{X}$ by the fully faithful assignment $\cM\mapsto (j_*\cM)^\mathrm{an}$. Therefore, the above functor \eqref{eq:SolEanalytic} induces an enhanced solution functor for an algebraic variety $X$
$$\cS ol^\mathrm{E}_{X_\infty}\colon \mathrm{D}^\mathrm{b}_\mathrm{hol}(\cD_X)^\mathrm{op}\hookrightarrow \mathrm{E}^\mathrm{b}(\mathrm{I}\CC_{X_\infty}),$$
where $\mathrm{E}^\mathrm{b}(\mathrm{I}\CC_{X_\infty})$ is a suitable subcategory of $\mathrm{E}^\mathrm{b}(\mathrm{I}\CC_{\widehat{X}^\mathrm{an}})$ and, in particular, does not depend on the choice of the smooth completion. We refer to \cite{Ito} for an exposition on the algebraic version of the algebraic enhanced Riemann--Hilbert correspondence (see also \cite{DHMS20} for the case of complex affine spaces).

The category $\mathrm{E}^\mathrm{b}(\mathrm{I}\CC_{X_\infty})$ comes equipped with a convolution product $\conv$ and direct and inverse image operations $\mathrm{E}f_*$, $\mathrm{E}f^{-1}$, $\mathrm{E}f_{!!}$ and $\mathrm{E}f^!$ for a morphism $f\colon X\to Y$ of smooth complex algebraic varieties. (These are denoted by $\mathrm{E}f^\mathrm{an}_{\infty*}$, etc.\ in \cite{Ito}.) Without recalling too many details, let us quickly recall the following: The category $\mathrm{E}^\mathrm{b}(\mathrm{I}\CC_{X_\infty})$ is a quotient of the derived category of ind-sheaves on $\widehat{X}^\mathrm{an}\times\PP^1(\RR)$, i.e.\ the base space is ``enhanced'' by some additional real variable. If $Y\subset X$ is a subvariety, there is a natural functor
$$\mathrm{E}^\mathrm{b}(\mathrm{I}\CC_{X_\infty})\to \mathrm{E}^\mathrm{b}(\mathrm{I}\CC_{X_\infty}), K\mapsto \pi^{-1}\CC_Y\otimes K,$$
where $\pi\colon X\times\RR\to X$ is the projection. This functor should be thought of as ``restriction to $Y$'' (or rather ``cutting off'' everything outside $Y$). An important class of objects are \emph{exponential enhanced ind-sheaves} $\EE^\varphi\in \mathrm{E}^\mathrm{b}(\mathrm{I}\CC_{X_\infty})$
for functions $\varphi\colon X\to \CC$. They are the topological counterpart via the functor $\cS ol^\mathrm{E}_{X_\infty}$ of exponential $\cD_X$-modules $\cE^\varphi$.

Finally, there is a fully faithful embedding
$$e\colon \mathrm{D}^\mathrm{b}(\CC_X)\hookrightarrow \mathrm{E}^\mathrm{b}(\mathrm{I}\CC_{X_\infty}),$$
enabling us to naturally consider complexes of sheaves (and in particular perverse sheaves) on $X$ as enhanced ind-sheaves. There is an isomorphism $\cS ol^\mathrm{E}_{X_\infty}(\cM)\iso e\cS ol_X(\cM)$ if $\cM$ is a regular holonomic $\cD_X$-module. We refer to \cite{DK16,DK19} for more details.

Classically, the Fourier--Laplace transform is the integral transform induced by the isomorphism of Weyl algebras $\mathsf{F}\colon \CC[z]\langle \partial_z\rangle\to \CC[w]\langle \partial_w\rangle, z\mapsto \partial_w, \partial_z\mapsto -w$. That is, if $M$ is a $\CC[w]\langle \partial_w\rangle$-module, its Fourier--Laplace transform $\frF(M)$ is the same underlying $\CC$-vector space with the action of $\CC[z]\langle\partial_z\rangle$ given by $z\cdot m\vcentcolon= \partial_w m$ and $\partial_z m\vcentcolon= -w\cdot m$. It was shown in \cite{KL} that this transform is given on (algebraic) $\cD_{\AA^1}$-modules by
$$\frF(\cM)={p_z}_*\big( \cE^{-wz} \otimes^\mathsf{D} p_w^* \cM \big),$$
where $p_w\colon \AA^1_w\times\AA^1_z\to\AA^1_w$ and $p_z\colon \AA^1_w\times\AA^1_z\to\AA^1_z$ are the projections and $\cE^{-wz}=\cD_{\AA^1_w\times\AA^1_z}/(\partial_w+z,\partial_z+w)$.
The inverse transform is described similarly, replacing $\cE^{-wz}$ by $\cE^{zw}$ and interchanging the two projections.

On the topological side, one can define the \emph{(inverse) enhanced Fourier--Sato transform}
$$K^\curlywedge \vcentcolon= \mathrm{E}{p_w}_{!!} \big( \EE^{zw} \conv \mathrm{E}p_z^{-1} K \big)\in \mathrm{E}^\mathrm{b}(\mathrm{I}\CC_{(\AA^1_w)_\infty})$$
for an object $K\in \mathrm{E}^\mathrm{b}(\mathrm{I}\CC_{(\AA^1_z)_\infty})$.\footnote{Note that the notation $K^\curlywedge$ is in accordance with \cite{DHMS20}, but the authors call this operation the (non-inverse) enhanced Fourier--Laplace transform in loc.\ cit. (since they work with the kernel $\cE^{zw}$ in the Fourier--Laplace transform). We will also often omit the word ``inverse''.}

We have the following result due to the nice compatibilities of the enhanced solution functor (see \cite[Theorem 9.4.10]{DK16}) and since exponential D-modules correspond to exponential enhanced ind-sheaves via the enhanced solution functor. (This was first observed in \cite{KS16}.)

\begin{lemma}\label{lemma:RH-Fourier}
Let $\cM\in\mathrm{D}^\mathrm{b}(\cD_{\AA^1_w})$. Then there is an isomorphism in $\mathrm{E}^\mathrm{b}(\mathrm{I}\CC_{(\AA^1_z)_\infty})$
$$\mathcal{S}ol^\mathrm{E}_{(\AA^1_z)_\infty}\big(\frF^{-1}(\cM)\big) \iso \mathcal{S}ol^\mathrm{E}_{(\AA^1_w)_\infty}(\cM)^\curlywedge[1].$$
\end{lemma}

\subsection{Enhanced Stokes phenomena}\label{sec:StokesPhenomena}
We briefly give an overview of the Stokes phenomenon, especially recalling how it is expressed in the context of enhanced solutions.

Let $\cM$ be an (algebraic) $\cD_{\AA^1}$-module with an irregular singularity at $\infty$ and let $z$ be a local coordinate of $\AA^1$ at $0$. It can be considered as an (analytic) $\cD_{\PP^1}$-module (where $\PP^1=\PP^1(\CC)$ is the complex projective line) satisfying $\cM(*\infty)\iso\cM$ (see \cite[p.\ 75]{Mal91}). Denote by $\widehat{\cO_{\PP^1,\infty}}$ the formalisation of the stalk $\cO_{\PP^1,\infty}$ and write $(\bullet)\widehat{|}_\infty\vcentcolon= \widehat{\cO_{\PP^1,\infty}}\otimes_{\cO_{\PP^1,\infty}}(\bullet)_\infty$. The formal classification of meromorphic connections in dimension $1$ states that there is a ramification map $\rho$ and an isomorphism
$$\big(\rho^*\cM\big)\widehat{\big|}_{\infty} \iso \big(\bigoplus_{i\in I} \cE^{\varphi_i}\otimes^\mathrm{D} \cR_i\big)\widehat{\big|}_{\infty}$$
for $I$ a finite index set, some $\varphi_i\in z\CC[z]$ and regular holonomic modules $\cR_i$.
This isomorphism is often referred to as the \emph{Levelt--Turrittin decomposition} of $\cM$ at $\infty$.

It is well-known that this decomposition generally does not lift to a decomposition of $\cM$ itself (without formalisation). However, it can be lifted on sufficiently small sectors on the real blow-up space of $\PP^1$ at $\infty$ (this is the Hukuhara--Turrittin theorem, see e.g.\ \cite[Théorème (1.4)]{Mal91}). This also yields a grading on small sectors of the Stokes-filtered local system associated to $\cM$ via the Riemann--Hilbert correspondence of Deligne--Malgrange. We refer e.g.\ to \cite{Bo20} for an overview of various descriptions of the Stokes phenomenon.

In terms of enhanced ind-sheaves, there is a similar statement: On any sufficiently small sector $S$ at $\infty$, there exists an isomorphism
\begin{equation} \label{eq:enhancedHT}
\pi^{-1}\CC_S\otimes \mathrm{E}\rho^{-1}\mathcal{S}ol^\mathrm{E}_{(\AA^1)_\infty}(\cM)\iso \pi^{-1}\CC_S\otimes \bigoplus_{i\in I} (\EE^{\varphi_i})^{r_i},
\end{equation}
where the $r_i$ are the ranks of the $\cR_i$. To be more precise, ``sufficiently small sector at $\infty$'' means that $S=\{z\in\AA^1\mid |z|> R, \arg z\in [\theta-\varepsilon,\theta+\varepsilon]\}$ for some $\theta\in\RR/2\pi\ZZ$ and some $R\geq 0$, $\varepsilon>0$.

In the case that all $\varphi_i$ as well as their pairwise differences have the same pole order $k$ (often called the case of \emph{pure level k}, see \cite{Mal83}, \cite{HS15}), one can describe the size of $S$ more concretely: There are $2k$ so-called \emph{Stokes directions} for every pair of exponents (they are the rays emanating from $\infty$ at which the asymptotic behaviour of $e^{\varphi_i-\varphi_j}$ changes). Any sector containing at most one Stokes direction for each pair of exponents, and hence in particular any sector of angle $\frac{\pi}{k}$ (and sufficiently small radius) whose boundary is not a Stokes direction, will admit a decomposition as in \eqref{eq:enhancedHT}. Hence, one can cover a neighbourhood of $\infty$ by $2k$ such sectors. If moreover the $\varphi_i$ are monomials of degree $k$ and $\cM$ has no other singularities except for a possible regular singularity at $0$, the radius of the sectors can be chosen to be infinite (i.e.\ $R=0$ above). 

This now yields Stokes matrices in the following way:
Let $\theta_0\in\RR/2\pi\ZZ$ be a \emph{generic direction} (i.e.\ not a Stokes direction for any pair of exponents $\varphi_i$, $\varphi_j$). Then consider the sectors
$$H_j=\left\{ z\in \AA^1\setminus \{0\} \,\Big| \arg z\in \left[\theta_0+(j-1)\frac{\pi}{k},\theta_0+j\frac{\pi}{k}\right] \right\}$$
for $j\in \{1,\ldots,2k\}$ and denote the decomposition isomorphisms by
$$\delta_j\colon \pi^{-1}\CC_{H_j}\otimes \mathrm{E}\rho^{-1}\mathcal{S}ol^\mathrm{E}_{(\AA^1)_\infty}(\cM)\overset{\sim}{\longrightarrow} \pi^{-1}\CC_{H_j}\otimes \bigoplus_{i\in I} (\EE^{\varphi_i})^{r_i}.$$
On the intersections of two adjacent sectors $H_{j,j+1}\vcentcolon= H_j\cap H_{j+1}$, one will then have two decompositions induced by $\delta_j$ and $\delta_{j+1}$, whose comparison yields transition matrices: The automorphisms $\sigma_j\vcentcolon=(\pi^{-1}\CC_{H_{j,j+1}}\otimes \delta_{j+1})\circ (\pi^{-1}\CC_{H_{j,j+1}}\otimes \delta_{j})^{-1}$ can be represented by invertible matrices $S_j$, the \emph{Stokes matrices} of $\cM$ at $\infty$ (see Fig.~\ref{fig:StokesExplanation} for an illustration). The generic direction $\theta_0$ defines a total ordering on the exponents $\varphi_i$, and with respect to this ordering the matrices representing $\sigma_j$ are upper (resp.\ lower) block-triangular if $j$ is odd (resp.\ even).
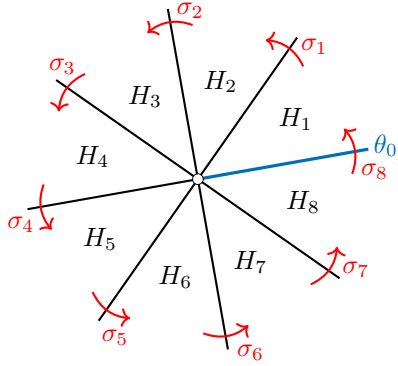
\begin{figure}[t]
	\centering
\begin{tikzpicture}
	\newcommand{\te}{10}
	\newcommand{\ra}{2.3}
	
	\draw[thick] ({-\ra*cos(\te)},{-\ra*sin(\te)})--(0,0);
	\draw[very thick, NavyBlue] (0,0)--({\ra*cos(\te)},{\ra*sin(\te)});
	\draw[thick] ({-\ra*cos(\te+45)},{-\ra*sin(\te+45)})--({\ra*cos(\te+45)},{\ra*sin(\te+45)});
	\draw[thick] ({-\ra*cos(\te+90)},{-\ra*sin(\te+90)})--({\ra*cos(\te+90)},{\ra*sin(\te+90)});
	\draw[thick] ({-\ra*cos(\te+135)},{-\ra*sin(\te+135)})--({\ra*cos(\te+135)},{\ra*sin(\te+135)});
	\node at (1.3,0.8) {$H_1$};
	\node at (-1.3,-0.8) {$H_5$};
	\node at (0.3,1.3) {$H_2$};
	\node at (-0.3,-1.3) {$H_6$};
	\node at (-0.7,1.1) {$H_3$};
	\node at (0.7,-1.1) {$H_7$};
	\node at (-1.4,0.3) {$H_4$};
	\node at (1.4,-0.3) {$H_8$};
	\node at (2.5,0.45) {\color{NavyBlue}$\theta_0$};
	
	\newcommand{\xn}{1.4}
	\newcommand{\yn}{1.5}
	\newcommand{\an}{0.9}
	\newcommand{\bn}{1.85}
	\draw [->,thick,red] (\xn,\yn) to [bend right] (\an,\bn);
	\node at (1.55,1.77) {\color{red}$\sigma_1$};
	\draw [->,thick,red] ({\xn*cos(45)-\yn*sin(45)},{\xn*sin(45)+\yn*cos(45)}) to [bend right] ({\an*cos(45)-\bn*sin(45)},{\an*sin(45)+\bn*cos(45)});
	\node at (-0.1,2.25) {\color{red}$\sigma_2$};
	\draw [->,thick,red] ({\xn*cos(90)-\yn*sin(90)},{\xn*sin(90)+\yn*cos(90)}) to [bend right] ({\an*cos(90)-\bn*sin(90)},{\an*sin(90)+\bn*cos(90)});
	\node at (-1.8,1.5) {\color{red}$\sigma_3$};
	\draw [->,thick,red] ({\xn*cos(135)-\yn*sin(135)},{\xn*sin(135)+\yn*cos(135)}) to [bend right] ({\an*cos(135)-\bn*sin(135)},{\an*sin(135)+\bn*cos(135)});
	\node at (-2.35,-0.6) {\color{red}$\sigma_4$};
	\draw [->,thick,red] ({\xn*cos(180)-\yn*sin(180)},{\xn*sin(180)+\yn*cos(180)}) to [bend right] ({\an*cos(180)-\bn*sin(180)},{\an*sin(180)+\bn*cos(180)});
	\node at (-1.1,-2.1) {\color{red}$\sigma_5$};
	\draw [->,thick,red] ({\xn*cos(225)-\yn*sin(225)},{\xn*sin(225)+\yn*cos(225)}) to [bend right] ({\an*cos(225)-\bn*sin(225)},{\an*sin(225)+\bn*cos(225)});
	\node at (0.7,-2.3) {\color{red}$\sigma_6$};
	\draw [->,thick,red] ({\xn*cos(270)-\yn*sin(270)},{\xn*sin(270)+\yn*cos(270)}) to [bend right] ({\an*cos(270)-\bn*sin(270)},{\an*sin(270)+\bn*cos(270)});
	\node at (2.1,-1.2) {\color{red}$\sigma_7$};
	\draw [->,thick,red] ({\xn*cos(315)-\yn*sin(315)},{\xn*sin(315)+\yn*cos(315)}) to [bend right] ({\an*cos(315)-\bn*sin(315)},{\an*sin(315)+\bn*cos(315)});
	\node at (2.35,0.1) {\color{red}$\sigma_8$};
	
	\draw (0,0) circle (2pt);
	\fill[White] (0,0) circle (1.8pt);
\end{tikzpicture}
\caption{Sectors determined by the choice of a generic direction $\theta_0$ and the transition isomorphisms in the case of exponents of the form  $\varphi_i=c_i z^4$. The point at the centre of the picture is the origin, and we consider the $H_j$ as sectors at $\infty$ with infinite radius.}
\label{fig:StokesExplanation}
\end{figure}

(The original idea for such a description of Stokes matrices using enhanced ind-sheaves is due to \cite[§9]{DK16}. We refer to \cite{DHMS20} for the case of pure level $1$ and to \cite[§§5--7]{Ho21} for monomials of degree $2$, noting that these arguments work accordingly for monomials of higher degree.)

\subsection{Perverse sheaves and quivers} \label{ss:pervquiv}

Let $\Sigma\subset \AA^1$ be finite. We recall how to describe perverse sheaves on $\AA^1$ with singularities at $\Sigma$ in terms of quivers. 
\begin{defn}
We denote by $\bQ_\Sigma$ the category whose objects are tuples
\[(\Psi, \Phi_s, u_s, v_s)_{s\in \Sigma}\]
where $\Psi$ and $\Phi_s$ are finite-dimensional $\CC$-vector spaces and $u_s:\Psi \rightarrow \Phi_s$, $v_s:\Phi_s \rightarrow \Psi$ are linear maps such that $T_s:=1-v_su_s$ is invertible for all $s\in \Sigma$.
\end{defn}

Any choice of $\fa, \fb \in \CC$ such that \begin{align}
 \Re(\fa \cdot \fb)&=0    \\
 \Re((s-s')\fb) &\neq 0 \quad \forall s,s'\in \Sigma, s\neq s'  \label{eq:betaCondition}
\end{align}
defines an equivalence of categories 
\[Q^{(\fa,\fb)}: \Perv_\Sigma(\AA^1) \rightarrow \bQ_\Sigma \]   
as described in \cite[§4]{DHMS20} and  \cite{GMV96}. For $F\in \Perv_\Sigma(\AA^1)$ with quiver $Q=(\Psi, \Phi_s, u_s, v_s)_{s\in \Sigma}$, the endomorphism $T_s$ of $\Psi$ is, up to conjugation, the local monodromy of $F$ at $s$.

\begin{defn} The {\it exponential dominance order} on $\Sigma$ is defined as follows. Fix a choice of $\fa,\fb$ as above. For $s, s'\in \Sigma$ we say that $s < s'$ if and only if $\Re(s\fb) < \Re(s'\fb)$. 
\end{defn}
Condition \eqref{eq:betaCondition} ensures that any two different elements of $\Sigma$ are comparable. We may then enumerate $\Sigma=\{s_1,\dots,s_n\}$ according to exponential dominance.

\subsection{Stokes multipliers for the enhanced Fourier--Sato transform} We quick\-ly recall how to obtain the Stokes multipliers of the enhanced Fourier--Sato transform $(eF)^\curlywedge$ for a perverse sheaf $F\in \Perv_\Sigma(\AA^1_z)$ following \cite[§5.2]{DHMS20}. With $\fa$ and $\fb$ as above, define
\begin{align*}H_{\pm \fa} &:= \{z \in \AA^1_w \mid \pm \Re(\fa w) \ge 0 \} \\
h_{\pm \fb} &:= \pm \RR_{>0} \fb \subset \AA^1_w. \end{align*}
Then $\AA^1_w = H_\fa \cup H_{-\fa}$ and $H_\fa \cap H_{-\fa}=h_\fb \cup h_{-\fb}$.

Let $K=(eF)^\curlywedge$ be the enhanced Fourier--Sato transform of $F$. The corresponding D-module is the Fourier--Laplace transform of the regular singular D-module associated to the perverse sheaf $F$. It is an irregular singular connection with pole of order $2$ at infinity -- its exponential factors are $e^{sw}$, $s\in\Sigma$. The $H_{\pm \fa}$ can be considered as sectors at $\infty$ containing exactly one Stokes direction for each pair of exponential factors so that there are decompositions
$$\delta_{\pm \fa}\colon \pi^{-1}\CC_{H_{\pm \fa}}\otimes K \overset{\sim}{\longrightarrow} \pi^{-1}\CC_{H_{\pm \fa}}\otimes \bigoplus_{s\in\Sigma} \EE^{sw}.$$
On the half-lines $h_{\pm\fb}$, one therefore has two such isomorphisms and the Stokes multipliers $S_{\pm\fb}$ are matrices representing the isomorphisms $(\pi^{-1}\CC_{h_{\pm\fb}}\otimes \delta_{-\fa})\circ(\pi^{-1}\CC_{h_{\pm\fb}}\otimes \delta_\fa)^{-1}$. (Note that this convention differs from the counter-clockwise arrangement of Stokes matrices in §\ref{sec:StokesPhenomena}.) For more details on the definition of Stokes multipliers in this setting we refer to \cite[§5.2]{DHMS20}. We recall the following result, which is a reformulation of some findings in \cite[Chap.\ XII]{Mal91} (cf.\ also \cite[§2]{Bo15}). We state it here in the form given in \cite{DHMS20} since this best fits our framework.

\begin{theorem}[{\cite[Theorem 5.2.2]{DHMS20}}]\label{theoremQuiverStokes}
Fix $\fa$ and $\fb$ as above. Let $F\in\Perv_\Sigma(\AA^1_z)$ be a perverse sheaf with associated quiver $(\Psi,\Phi_s,u_s,v_s)_{s\in\Sigma}$. 
Choose a numbering of $\Sigma$ such that $s_1<\ldots<s_n$ and write $u_k\vcentcolon=u_{s_k}$, $v_k\vcentcolon=v_{s_k}$ and set $\TT_k\vcentcolon= 1-u_kv_k$.

Denote by $K\vcentcolon= (eF)^\curlywedge\in \mathrm{E}^\mathrm{b}(\mathrm{I}\CC_{\AA^1_w})$ the associated enhanced Fourier--Sato transform. The Stokes multipliers $S_{\pm \fb}$ of $K$ at $\infty$ are given by the block-triangular matrices
\begin{align}
    S_\fb &= \begin{pmatrix} 1 & u_1v_2 & u_1v_3 &  \dots & u_1v_n \\
     & 1 & u_2v_3 & \dots & u_2v_n  \\
     & & \ddots &  & \vdots \\
     & & & & 1 
    \end{pmatrix} \\
    S_{-\fb} &= \begin{pmatrix} \TT_1 &  &  &  \\
     -u_2v_1 & \TT_2 &  &   \\
     \vdots &   & \ddots &  \\
     -u_nv_1 & -u_nv_2  & \dots &  \TT_n 
    \end{pmatrix}.
\end{align}
\end{theorem}

\section{Stokes matrices for Airy equations}\label{sect:general}
\subsection{Generalised Airy equations}
We will write $\partial_y\vcentcolon= \frac{d}{dy}$ and denote by $\CC[y]\langle\partial_y\rangle$ the Weyl algebra.

In \cite{Ka87}, N.\ Katz introduced the class of \emph{differential operators of Airy type} on $\AA^1$, which are operators of the form
$$P_n(\partial_y) + Q_m(y)\in \CC[y]\langle \partial_y\rangle,$$
where $P_n,Q_m\in \CC[y]$ are polynomials of degree $n$ and $m$, respectively.

The classical Airy operator is recovered by choosing $n=2$, $m=1$ and $P_2(y)=y^2$, $Q_1(y)=y$.

The $\cD_{\AA^1}$-module associated to such an operator has an irregular singularity at $\infty$ of slope $\frac{n+m}{n}$. We will compute the Stokes matrices at $\infty$ when $P_n$ and $Q_m$ are monomials, i.e.\ for operators of the form
$$\partial_y^n - y^m.$$
If $m=1$, the Fourier--Laplace transform of this operator has rank $1$ and hence the equation defined by it is rigid in the sense that its formal type determines its isomorphism class uniquely. One may check using the index of rigidity \cite[Definition 4.2.]{BE04} that in fact the equation corresponding to $\partial_y^n - y^m$ is rigid if and only if $n=1$ or $m=1$.
\subsection{Airy equations and hypergeometrics}\label{subsec:AiryHyp} In this section, we are going to investigate the module $\cM_{n,m}=\cD_{\AA^1_y}/\cD_{\AA^1_y}P_{n,m}$ associated to the operator
$$P_{n,m}=\partial_y^n - y^m\in\CC[y]\langle\partial_y\rangle.$$
Our aim is to describe a procedure for determining its Stokes multipliers at $\infty$, and we want to reduce this question to the question of describing a regular holonomic D-module. More precisely, similarly to \cite[Lemma 7.1.3]{DHMS20}, we want to describe $\cM_{n,m}$ (in a neighbourhood of $\infty$ and after pulling back via a suitable ramification map) by a pullback along a covering map of a Fourier--Laplace transform of a D-module with regular singularities.

Since we want to understand the behaviour of $\cM_{n,m}$ around $\infty$, it suffices to consider this module restricted to $\GM{y}\subset\AA^1_y$. We have isomorphisms
\begin{align}\label{eq:AiryModuleOnGm}
\cM_{n,m}\big|_{\GM{y}}=\cD_{\GM{y}}/\cD_{\GM{y}}P_{n,m}&\iso \cD_{\GM{y}}/\cD_{\GM{y}}\big(y(y^n\partial_y^n-y^{n+m})\big)\\
&\iso \cD_{\GM{y}}/\cD_{\GM{y}}(Qy)\notag\\
&\iso \cD_{\GM{y}}/\cD_{\GM{y}}Q\notag
\end{align}
with the operator
$$Q=\prod_{k=1}^n (y\partial_y-k) - y^{n+m}.$$
The first isomorphism in \eqref{eq:AiryModuleOnGm} follows from the fact that multiplication by $y$ is invertible on $\GM{y}$.
Since $y^n\partial_y^n=\prod_{k=0}^{n-1} (y\partial_y-k)$, the second isomorphism is easily deduced from the relation $y(y\partial_y-k)=(y\partial_y-(k+1))y$. The last isomorphism follows from the short exact sequence
\begin{equation}\label{eq:sequenceOnGm}
0\longrightarrow \cD_{\GM{y}}/\cD_{\GM{y}}Q\longrightarrow \cD_{\GM{y}}/\cD_{\GM{y}}(Qy)\longrightarrow \cD_{\GM{y}}/\cD_{\GM{y}}y \longrightarrow 0,
\end{equation}
noting that $\cD_{\GM{y}}/\cD_{\GM{y}}y\iso 0$.

We now pull back this module along the ramification map $\rho_n\colon \AA^1_v\to \AA^1_y, v\mapsto v^n$. This amounts to a substitution $y=v^n$, $y\partial_y=\frac{1}{n}v\partial_v$ in the operator $Q$, hence we obtain
$$\rho_n^* \cM_{n,m}\big|_{\GM{v}}\iso \cD_{\GM{v}}/\cD_{\GM{v}}L, \qquad L=\prod_{k=1}^n \left(\frac{1}{n}v\partial_v-k\right) - v^{n(n+m)}.$$

It is easy to check that this coincides with the pullback along the covering map $g\colon \AA^1_v\to\AA^1_w, v\mapsto \frac{v^{n+m}}{n+m}$ of the D-module associated with the operator
$$S=\prod_{k=1}^n \left(\frac{n+m}{n}w\partial_w-k\right) - (n+m)^nw^n$$
(obtained from $L$ via the substitution $v^{n+m}=(n+m)w$, $v\partial_v=(n+m)w\partial_w$).
Furthermore, considering the isomorphism of Weyl algebras $\mathsf{F}\colon \CC[z]\langle\partial_z\rangle\to \CC[w]\langle\partial_w\rangle$, $z\mapsto \partial_w, \partial_z\mapsto -w$, we can write
\begin{align*}S=\mathsf{F}(R), \qquad R&=\prod_{k=1}^n \left(\frac{n+m}{n}(-1-z\partial_z)-k\right) - (-1)^n(n+m)^n\partial_z^n\\
&=(-1)^n(n+m)^n\left(\frac{1}{n^n}\prod_{k=1}^n \left(z\partial_z+1+\frac{kn}{n+m}\right) - \partial_z^n\right),\end{align*}
which implies that
$$\rho_n^*\cM_{n,m}\big|_{\GM{v}}\iso g^* \frF^{-1}\big(\cD_{\AA^1_z}/\cD_{\AA^1_z}R\big)\big|_{\GM{v}}.$$

The operator $R$ defines a regular holonomic D-module, which is the pullback of a hypergeometric system along a ramification map:
First, note that
$$(-1)^n\frac{1}{(n+m)^n} z^{n+1} R = \left(\frac{z}{n}\right)^n z\prod_{k=1}^n \left(z\partial_z+1+\frac{kn}{n+m}\right) - z\prod_{k=0}^{n-1} (z\partial_z - k)=\widetilde{R}z$$
with
$$\widetilde{R}= \left(\frac{z}{n}\right)^n\prod_{k=1}^n \left(z\partial_z+\frac{kn}{n+m}\right) - \prod_{k=1}^{n} (z\partial_z - k).$$
An exact sequence as in \eqref{eq:sequenceOnGm} shows that
$$\cD_{\GM{z}}/\cD_{\GM{z}}R\iso \cD_{\GM{z}}/\cD_{\GM{z}}(\widetilde{R}z)\iso \cD_{\GM{z}}/\cD_{\GM{z}}\widetilde{R}.$$
By \cite[Lemma 2.9.5]{Ka90}, the D-modules on $\AA^1_z$ defined by $R$ and $\widetilde{R}$ only depend on their restriction to $\GM{z}$, more precisely we have $j_!\cD_{\GM{z}}/\cD_{\GM{z}}R\iso \cD_{\AA^1_z}/\cD_{\AA^1_z}R$ and $j_!\cD_{\GM{z}}/\cD_{\GM{z}}\widetilde{R}\iso \cD_{\AA^1_z}/\cD_{\AA^1_z}\widetilde{R}$, where $j\colon \GM{z}\hookrightarrow\AA^1_z$ is the inclusion.
Considering the covering map $\gamma_n\colon \AA^1_z \to \AA^1_x, z\mapsto \frac{z^n}{n^n}$ (pulling back along which corresponds to the substitution $\frac{z^n}{n^n}=x$, $z\partial_z=nx\partial_x$), we therefore obtain
\begin{align*}&\quad\cD_{\AA^1_z}/\cD_{\AA^1_z}R\iso \cD_{\AA^1_z}/\cD_{\AA^1_z}\widetilde{R}\iso \gamma_n^*\big(\cD_{\AA^1_z}/\cD_{\AA^1_z}H\big),\\ &H=n^nx\prod_{k=1}^n \left(x\partial_x+\frac{k}{n+m}\right) - n^n\prod_{k=1}^{n} \left(x\partial_x - \frac{k}{n}\right).
\end{align*}

The operator $H$ now defines a hypergeometric equation, and we summarise the outcome of our above computations in the following proposition.
\begin{prop}\label{prop:hyptoairy}
Let $P_{n,m}=\partial_y^n-y^m$, then we have an isomorphism of holonomic algebraic $\cD_{\GM{v}}$-modules (with the morphisms as above)
$$\rho_n^*\cM_{n,m}\big|_{\GM{v}}\iso g^*\frF^{-1}\left(\gamma_n^*\left(\cD_{\AA^1_x}/\cD_{\AA^1_x}\mathrm{Hyp}(\alpha,\beta)\right)\right)\big|_{\GM{v}},$$
where
$$\mathrm{Hyp}(\alpha,\beta)=\prod_{k=1}^n(x\partial_x-\alpha_k)- x \prod_{k=1}^n(x\partial_x-\beta_k)$$
is the hypergeometric operator of type $(n,n)$ with
\begin{align*}
    \alpha_k = \frac{k}{n},\qquad \beta_k = -\frac{k}{n+m}.
\end{align*}
\end{prop}
Applying the enhanced solution functor on both sides and using its nice compatibilites with inverse images, Fourier--Laplace transforms as well as the classical solution functor for regular holonomic D-modules (cf.\ §\ref{subsec:enhanced}), one obtains the following corollary. It is the starting point for determining the Stokes multipliers by the results of \cite{DHMS20} (cf.\ Theorem~\ref{theoremQuiverStokes}).

\begin{cor}
There is an isomorphism in $\mathrm{E}^\mathrm{b}(\mathrm{I}\CC_{\AA^1_\infty})$
$$\pi^{-1}\CC_{\GM{v}}\otimes \mathrm{E}\rho_n^{-1} \cS ol^\mathrm{E}_{(\AA^1_y)_\infty}(\cM_{n,m})\iso \pi^{-1}\CC_{\GM{v}}\otimes \mathrm{E}g^{-1}((eF)^\curlywedge),$$
where
$$F=\gamma_n^*\cS ol_{\AA^1_x}\big(\cD_{\AA^1_x}/\cD_{\AA^1_x}\mathrm{Hyp}(\alpha,\beta)\big)[1].$$
In particular, in view of Proposition~\ref{theoremQuiverStokes}, the Stokes matrices at $\infty$ of $\cM_{n,m}$ only depend on the quiver associated to the perverse sheaf $F$.
\end{cor}
\subsection{The hypergeometric equation.} Let $\cH=j^*\cD_{\AA^1_x}/\cD_{\AA^1_x}\mathrm{Hyp}(\alpha,\beta)$ denote the hypergeometric connection with $j: \AA^1_x \setminus \{0,1\} \rightarrow \AA^1_x$ defined by the operator $\mathrm{Hyp}(\alpha, \beta)$ from above. We will describe the monodromy of $\cH$ explicitly to determine the quiver of the perverse sheaf $F$. \par

The fundamental group  $\pi_1(\AA^1_x\setminus \{0,1\}, x_0)$ is generated by simple counter-clock\-wise loops $g_0, g_1, g_\infty$ around $0,1,\infty$ and based at $x_0$ such that 
\[g_0g_1g_\infty = 1. \]
Let
$\rho: \pi_1(\AA^1_x\setminus \{0,1\}, x_0)\rightarrow \GL(V)$
be the monodromy representation of $\cH$ (i.e.\ $V=\cV_{x_0}$ for $\cV$ the local systems of solutions of $\cH$) and $T_i=\rho(g_i)$ for $i\in \{0,1,\infty\}$. Then $T_0$ and $T_\infty$ are semisimple with eigenvalues
\begin{align*}\exp(2\pi i\alpha_k) \,\;\textup{and}\;
\exp(-2\pi i\beta_k) \;\; \textup{for} \; k=1,\dots,n
\end{align*}
respectively. It is well-known that $T_1$ is a pseudo-reflection with non-trivial eigenvalue 
\[\exp(2\pi i\gamma),\; \gamma = \sum_{k=1}^n(\beta_k - \alpha_k ).\]
Note that in our case $T_1$ has finite order and is actually a complex reflection. The monodromy representation $\rho$ may be described explicitly by the following rigidity result of Levelt \cite[Theorem 3.5]{BH89}. 
\begin{theorem}\label{thm:levelt} Let $a_1,\dots,a_n,b_1,\dots, b_n\in \CC^\times$ such that $a_j\neq b_k$ for all $j,k$. Write 
\begin{align*}\prod_{i=1}^n (X-a_i) = X^n+A_1X^{n-1}+\dots+A_n, \\ \prod_{i=1}^n (X-b_i) = X^n+B_1X^{n-1}+\dots+B_n. \end{align*}
Given a complex vector space $V$ and invertible linear transformations $A,B\in \GL(V)$ with eigenvalues $a_i$ and $b_i$, respectively, and such that $AB^{-1}$ is a pseudo-reflection, then there is a choice of basis of $V$ such that in that basis $A$ and $B$ are given by companion matrices 
\[\begin{pmatrix} 0 & & & & -A_n \\
1 & 0 & & & -A_{n-1} \\
 & \ddots & \ddots & & \vdots \\
 & & & 0 & -A_2 \\
 & & & 1 & -A_1 
\end{pmatrix}, \, \begin{pmatrix} 0 & & & & -B_n \\
1 & 0 & & & -B_{n-1} \\
 & \ddots & \ddots & & \vdots \\
 & & & 0 & -B_2 \\
 & & & 1 & -B_1 
\end{pmatrix}.
\]
\end{theorem}
Since in our case $\alpha_i = \frac{i}{n}$, the eigenvalues of $T_0$ are precisely the $n$-th roots of unity. Therefore, up to isomorphism, we may assume that $T_0$ is a cyclic permutation matrix corresponding to the $n$-cycle $(n \dots 2 \, 1)$, i.e. the companion matrix of $X^n-1$. \par
\subsection{Monodromy of the pullback} \label{sub:localmon} The computations in this section are similar to those in \cite{H20}, Section 7. In our situation, pulling back along $\gamma_n$ trivialises the local monodromy at $z=0$, allowing us simplify the argument by extending the local system across the origin.
Let $\mu_n\subset \Gm$ denote the set of $n$-th roots of unity and set $\zeta_n\vcentcolon=\exp(\frac{2\pi i}{n})$. Let \[U=\AA^1_z \setminus ( n\mu_n \cup \{0\}).\] The map $\gamma:=\gamma_n|_{U} : U \rightarrow \AA^1_x\setminus \{0,1\}$ is a cyclic Galois cover with Galois group $\mu_n$. Let $\cV$ denote the local system of solutions of $\cH$. Then $\gamma^*\cV$ is $\mu_n$-equivariant. The following lemma is easy to check.

\begin{lemma}\label{lem:equivcov}
Let $X$ be a topological space with an action of a finite group $G$. Let $x\in X$ be a fixed point of the $G$-action and $\cV$ a $G$-equivariant local system on $X$ with fibre $V$ over $x$ and monodromy representation
\[\rho:\pi_1(X,x)\rightarrow \GL(V).\]
Then $G$ acts on $V$ and we have
\[\rho(g_*(\gamma))(v)=g(\rho(\gamma)(g^{-1}v)) \]
for all $v\in V$, $g\in G$ and $\gamma\in \pi_1(X,x)$, where $g_*$ denotes the automorphism of $\pi_1(X,x)$ induced by the action map. 
\end{lemma}
\begin{prop}
The local system $\gamma^*\cV$ extends to a $\mu_n$-equivariant local system on $\AA^1_z\setminus n\mu_n$. Its local monodromy at $n\zeta_n^k\in n\mu_n$ is equal to $T_0^kT_1T_0^{-k}$. 
\end{prop}
\begin{proof} For $u\in U$ with $\gamma(u)=x_0$, the map $\gamma$ induces an injective homomorphism 
\[\pi_1(U,u) \xrightarrow{\gamma_*} \pi_1(\AA^1_x\setminus \{0,1\},x_0) \]
and the monodromy representation of $\gamma^*\cV$ is obtained by restricting $\rho$ along this homomorphism. We denote this restriction by $\rho':=\rho\circ\gamma_*$. Since $\gamma^*\cV$ has trivial local monodromy at $0$, it extends uniquely to $\AA^1_z\setminus n\mu_n$, and since $0$ is fixed by the $\mu_n$-action, the extension is $\mu_n$-equivariant. It follows from standard theory of covering spaces that the action of the generator $\zeta_n$ of the Galois group $\mu_n$ on $V=(\gamma^*\cV)_u$ agrees with the action of the local monodromy $T_0$ of $\cV$ at $0$. \par 
Choose a set of generators $g_j'$ of $\pi_1(\AA^1_z\setminus n\mu_n,u)$ such that $g_j'$ is a simple counter-clockwise loop around $n\zeta_n^j$ and such that $g'_\infty g'_n\dots g'_1=1$. \par
Note that $\mu_n$ acts on $\pi_1(\AA^1_z \setminus n\mu_n,0)$ by $(\zeta_* g)(t)=\zeta g(t)$. Choose a small simple counter-clockwise loop $g''_1$ based at $0$ such that $\pi_1(\AA^1_z \setminus n\mu_n,0)$
is generated by $g''_1, \zeta_ng''_1, \dots, \zeta_n^{n-1}g''_1$.  We may choose an isomorphism 
\[\begin{tikzcd}
\pi_1(\AA^1_z\setminus n\mu_n,0) \ar{r}[swap]{\cong}{\phi}  & \pi_1(\AA^1_z\setminus n\mu_n,u)
\end{tikzcd}\]
such that $\phi(\zeta_n^{j}g''_1)=g_j'$. Define $\mu:=\rho'\circ\phi$. By Lemma \ref{lem:equivcov} we know that
\[\mu(\zeta_{n,*}(g))(v) = \zeta_n.\mu(g)(\zeta_n^{-1}.v)\]
for any $v\in V$ and $g\in \pi_1(\AA^1_z\setminus n\mu_n,0)$.
This implies that
\[\rho'(g_{j+1}') = T_0\rho'(g_j')T_0^{-1}. \]
To check that $\gamma_*(g_1')$ is homotopic to $g_1\in \pi_1(\AA^1_x \setminus \{0,1\}, x_0)$, one may choose representatives for $g_1'$ and $g_1$ and write down an explicit homotopy. The claim follows. 
\end{proof}
Up to isomorphism we therefore obtain an explicit description of the local monodromy of the perverse sheaf $F$ in terms of conjugates of a single companion matrix. 
\begin{remark}
In the terminology of Beukers--Heckmann \cite[§5]{BH89} the group generated by the $T_0^k T_1 T_0^{-k}$ is the reflection subgroup of the hypergeometric group corresponding to $\cH$. The shape of the local exponents implies that this group acts irreducibly on $V$. If this group is finite, it is a complex reflection group. The complex reflection groups that may occur as reflection subgroups of hypergeometric groups are listed in \cite[§8]{BH89}.
\end{remark}

\subsection{Stokes multipliers for generalised Airy equations} Recall that $F=\gamma_n^*\cS ol_{\AA^1_x}\big(\cD_{\AA^1_x}/\cD_{\AA^1_x}\mathrm{Hyp}(\alpha,\beta)\big)[1]$. 
\begin{lemma} Denoting by $k:\AA^1_x \setminus n\mu_n \hookrightarrow \AA^1_x$ the open embedding, we have
\[k_{!*}k^* F \cong F, \]
i.e.\ $F$ is the minimal extension of its restriction to $\AA^1_x \setminus n\mu_n$. 
\end{lemma}
\begin{proof}
Since $n$ and $m$ are coprime, the hypergeometric $\cD$-module $$\cD_{\GM{x}}/\cD_{\GM{x}}\mathrm{Hyp}(\alpha,\beta)$$ is irreducible on $\Gm$, cf.\ \cite[Corollary 3.2.1]{Ka90}. By \cite[Corollary 2.9.6.1]{Ka90} we have \[\cD_{\GM{x}}/\cD_{\GM{x}}\mathrm{Hyp}(\alpha,\beta) \cong h_{!*}\cH \]
where $h:\GM{x}\setminus \{1\} \rightarrow \GM{x}$ is the open embedding. Therefore the same holds for the perverse sheaf of solutions $H$ on $\GM{x}$. The minimal extension is characterised by the property that $H$ has no non-zero quotients or subobjects supported on $\{1\}$. Now let $z\in n\mu_n$ and choose a neighbourhood $U$ of $z$ such that 
\[\gamma|_U : U \rightarrow \gamma(U) \]
is a local isomorphism with inverse $\widetilde{\gamma}$. To prove the assertion it is enough to show that $F$ has no subobject or quotient supported at $z$. Assume $G_z$ is either a subobject or quotient of $F$ supported at $z$. Then $\widetilde{\gamma}^*(G_z|_U)$ is a subobject or quotient of $H|_{\gamma(U)}$ supported at $\widetilde{\gamma}(z)=1$ and therefore $G_z=0$. This proves the claim. 
\end{proof}
To describe the quiver of $F$ we need to choose $\fa$ and $\fb$ as in Section \ref{ss:pervquiv} and describe the exponential dominance order on the set of singularities $n\mu_n$ of $F$. This is best done geometrically: 
Let $\fb=\exp(\frac{2\pi i}{3n})$ and accordingly we choose $\fa=\frac{i}{\fb}$. We then look at the set $\fb \mu_n$ and order it by the size of the real parts. This order may be visualised as a zig-zag from right to left (see Fig.~\ref{fig:ordering}), i.e. for $\zeta=\zeta_n=\exp(\frac{2\pi i}{n})$ we have
\[1 > \zeta^{n-1} > \zeta > \zeta^{n-2} > \zeta^2 > \dots \]

\begin{figure}[t]
	\centering
	\begin{tikzpicture}[scale=1.5]
		\node at (-1.25,1.25) {\color{Green}$\mu_n$};
		\draw[->,Gray,thick] (-1.5,0) -- (1.5,0);
		\node at (1.7,-0.2) {\small\color{Gray}$\mathrm{Re}\, z$};
		\draw[->,Gray,thick] (0,-1.5) -- (0,1.5);
		\node at (-0.2,1.7) {\small\color{Gray}$\mathrm{Im}\, z$};
		\foreach \x in {-2,...,2} {
			\draw[Gray,opacity=0.5] (\x*0.5,-1.5)--(\x*0.5,1.5);
			\draw[Gray,opacity=0.5] (-1.5,\x*0.5)--(1.5,\x*0.5);
		}
		\foreach \n in {0,...,6} {
			\fill[Green] ({cos(360/7*\n)},{sin(360/7*\n)}) circle (1.3pt);
			\ifthenelse{\n=0}{\draw[Red,thick] ({cos(360/7*\n)},{sin(360/7*\n)}) circle (1.4pt);}{}
		}
		\node at (1.15,0.15) {\small \color{Green}$1$};
		\node at (0.78,0.8) {\small\color{Green} $\zeta$};
		\node at (0.93,-0.85) {\small\color{Green} $\zeta^{n-1}$};
	\end{tikzpicture}\hspace{1cm}
	\begin{tikzpicture}[scale=1.5, decoration={markings, 
			mark= at position 0.5 with {\arrow{angle 60}}}] 
		\node at (-1.25,1.25) {\color{Green}$\fb\mu_n$};
		\draw[->,Gray,thick] (-1.5,0) -- (1.5,0);
		\node at (1.7,-0.2) {\small\color{Gray}$\mathrm{Re}\, z$};
		\draw[->,Gray,thick] (0,-1.5) -- (0,1.5);
		\node at (-0.2,1.7) {\small\color{Gray}$\mathrm{Im}\, z$};
		\foreach \x in {-2,...,2} {
			\draw[Gray,opacity=0.5] (\x*0.5,-1.5)--(\x*0.5,1.5);
			\draw[Gray,opacity=0.5] (-1.5,\x*0.5)--(1.5,\x*0.5);
		}
		\foreach \n in {0,...,6} {
			\ifthenelse{\n<4}{\draw[dotted,thick,postaction={decorate}] ({cos(360/7*\n+360/7/3)},{sin(360/7*\n+360/7/3)})--({cos(360/7*(-\n-1)+360/7/3)},{sin(360/7*(-\n-1)+360/7/3)});}{\draw[dotted,thick,postaction={decorate}] ({cos(360/7*\n+360/7/3)},{sin(360/7*\n+360/7/3)})--({cos(360/7*(-\n)+360/7/3)},{sin(360/7*(-\n)+360/7/3)});}
		}
		\foreach \n in {0,...,6} {
			\fill[Green] ({cos(360/7*\n+360/7/3)},{sin(360/7*\n+360/7/3)}) circle (1.3pt);
			\ifthenelse{\n=0}{\draw[Red,thick] ({cos(360/7*\n+360/7/3)},{sin(360/7*\n+360/7/3)}) circle (1.4pt);}{}
		}
		\node at (1.1,0.33) {\small \color{Green}$\fb$};
		\node at (0.57,1.05) {\small\color{Green} $\fb\zeta$};
		\node at (1.1,-0.75) {\small\color{Green} $\fb\zeta^{n-1}$};
	\end{tikzpicture}
	\caption{A total ordering on the set $\mu_n$ of $n$-th roots of unity (shown on the left-hand side) is obtained by multiplying them with $\fb=\exp(\frac{2\pi i}{3n})$ and ordering them by the size of their real parts, which results in the zig-zag scheme indicated on the right-hand side. The picture illustrates the case $n=7$, but it is easily checked that this procedure works for any $n\in\mathbb{Z}_{>0}$.}
	\label{fig:ordering}
\end{figure}
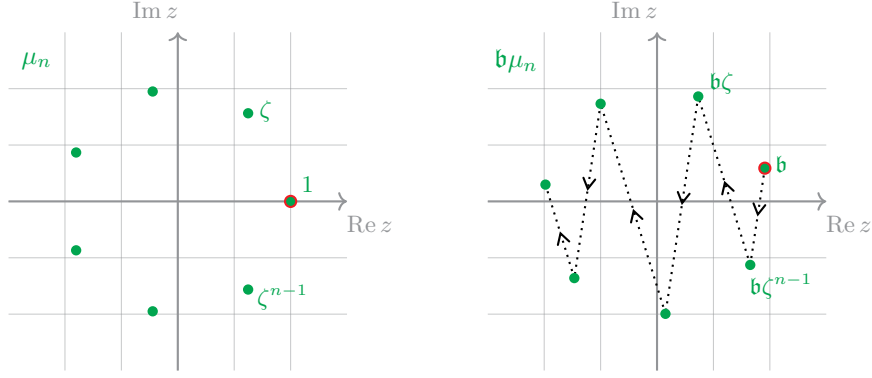

The following lemma is adapted from  \cite[Corollary 4.2.]{H20}. It is the standard description of the quiver associated to a minimal extension sheaf. 
\begin{lemma}
Let $D$ be a disk in $\CC$, $\sigma \in D$ and $D^\times = D\setminus \{\sigma\}$ with embedding $i: D^\times \rightarrow D$. Let $K$ be a perverse sheaf on $D$ such that $i^*K[-1]=L$ is a local system. Then $K\cong i_{!*}i^*K$ if and only if its quiver is isomorphic to 
\[\begin{tikzcd} \Psi(K) \arrow[r, shift left=1, "1-T_\sigma"]& \arrow[l, shift left=1, "\iota_\sigma"] \Phi_\sigma(K)= \mathrm{im}(1-T_\sigma),
\end{tikzcd}\]
where $T_\sigma$ is the local monodromy of $L$ at $\sigma$ and $\iota_\sigma$ is the natural embedding. 
\end{lemma}

\begin{cor} \label{cor:circlequiv}
The quiver of $F$ has $n$ nodes, one for each $z\in n\mu_n$. Let $P=T_0$ be the local monodromy of $\cH$ at $0$ considered as an automorphism of $\Psi(F)$. For $\zeta_n=\exp(\frac{2\pi i}{n})$ the nodes at $z$ and $\zeta_nz$ look as follows: 
\[\begin{tikzcd}
 \mathrm{im}(1-T_z)=\Phi_z(F) \arrow[shift right=1]{rr}[swap]{\iota_z} &  & \arrow[ shift right=1]{ll}[swap]{1-T_z} \Psi(F) \arrow[rr, shift left=1, "1-PT_zP^{-1}"] & & \arrow[ll, shift left=1, "P\iota_zP^{-1}"] \Phi_{\zeta_n z}(F)= P(\mathrm{im}(1-T_z))
\end{tikzcd}\]

\end{cor}
\begin{proof} This follows immediately from the preceding lemma and the computation of the local monodromy of $\gamma^*\cH$ in Section \ref{sub:localmon}.
\end{proof}

\begin{theorem} \label{thm:stokesreg} 
Let $\prod_{j=1}^n (X-\zeta_{n+m}^{-j}) = X^n+\lambda_1X^{n-1}+\dots+\lambda_{n-1}X+\lambda_n$. Then the Stokes matrices for the generalised Airy equation $\rho_n^*\cM_{n,m}$ are given by the $(n \times n)$-matrices 
\[S_{2k} = 
\begin{pmatrix} 1 & {s_{12}} & {s_{13}} &  \dots & {s_{1n}} \\
     & 1 & {s_{23}} & \dots & {s_{2n}}  \\
     & & \ddots & \ddots & \vdots \\
     & & & \ddots & {s_{n-1,n}}\\
     & & & & 1 
\end{pmatrix},\quad S_{2k-1}^{-1} = 
\begin{pmatrix} -\lambda_n &  &  &  & \\
    {-s_{21}} & -\lambda_n &  & & \\
    \vdots & & \ddots &  &  \\ 
    & & & \ddots &\\
    {-s_{n,1}} & \dots & & &  -\lambda_n 
\end{pmatrix}\] 
for certain $s_{ij}\in \{\lambda_1,\dots, \lambda_{n-1}\}$, $k=1,\dots,n+m$ and with respect to the generic direction $\theta_0=\frac{2\pi}{3n(n+m)}$. Moreover, the unipotent Stokes matrices are regular in the sense that their centraliser dimension is minimal.
\end{theorem} 
\begin{proof}
To compute the Stokes matrices, we need to describe the quiver of $F$ as explicitly as possible. The monodromy of $\cH$ at $1$ is the pseudo-reflection \[T_1=T_0^{-1}T_\infty^{-1}\]
with non-trivial eigenvalue $-\lambda_n$. 
We choose a basis such that $T_0$ and $T_\infty^{-1}$ are given as companion matrices. Then $T_0^{-1}$ is the permutation matrix given by the $n$-cycle $(1\, 2\, \dots \, n)$ and $T_\infty^{-1}$ is the matrix

\[\begin{pmatrix} 
   &           &       & -\lambda_n \\
1   &          &       & \vdots  \\
    & \ddots    &       & -\lambda_2 \\
    &           & 1     & -\lambda_1  \\
    \end{pmatrix}.\]
Furthermore, $1-T_1$ is a matrix in which all entries apart from the last column vanish and the last column is the transpose of 
\[(\lambda_{n-1}, \dots, \lambda_1, 1+\lambda_n). \]
Let $\ell$ and $k$ be integers between $1$ and $n$. By Corollary \ref{cor:circlequiv} the quiver of $F$ looks as follows 

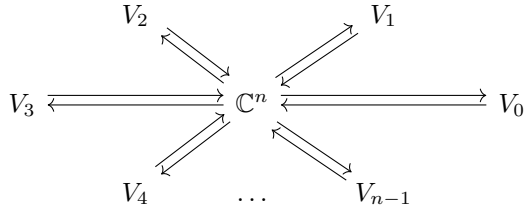
\[\begin{tikzcd}
& V_2 \arrow[dr, shift left=1.5]  &   &  V_1 \ar{dl} \\
V_3\arrow[rr, shift left=1.5] & & \CC^n \arrow[ur, shift left=1.5] \ar{ul} \ar{ll} \ar{dl} \arrow[rr, shift left=1.5] \ar{dr} & &\ar{ll} V_0 \\
 & V_4 \arrow[ur, shift left=1.5] & \dots  & \arrow[ul, shift left=1.5]V_{n-1} & \\
\end{tikzcd}
\]
with $V_j=P^j\mathrm{im}(1-T_{1})$. Thus the part relating the singularities $n\zeta^k$ and $n\zeta^\ell$ is given by

\[\begin{tikzcd}
 P^k\mathrm{im}(1-T_{1}) \arrow[shift right=1]{rr}[swap]{} &  & \arrow[ shift right=1]{ll}[swap]{1-P^{k}T_{1}P^{-k}} \CC^n \arrow[rr, shift left=1, "1-P^{\ell}T_{1}P^{-\ell}"] & & \arrow[ll, shift left=1]  P^\ell\mathrm{im}(1-T_{1}),
\end{tikzcd}\]
with the lower horizontal arrows being the natural inclusions. Here $P=T_0$ is a cyclic permutation matrix corresponding to the $n$-cycle $(n \, \dots \, 1)$. The image of $1-T_{1}$ is one-dimensional and generated by the non-trivial column above. Only the $(n-\ell)$-th column of the matrix $1-P^{\ell}T_{1}P^{-\ell}$ is non-trivial. Therefore the linear map 
\[P^k\mathrm{im}(1-T_{1}) \hookrightarrow \CC^n  \xrightarrow{1-P^{\ell}T_{1}P^{-\ell}} P^\ell\mathrm{im}(1-T_{1}) \]
is given by multiplication with the $(n-\ell)$-th entry of $P^k(\lambda_{n-1}, \dots, \lambda_1, 1+\lambda_n)^T$. For $k\neq \ell$ this entry is not $1+\lambda_n$. By Theorem \ref{theoremQuiverStokes} the upper-triangular Stokes matrix of $(eF)^\curlywedge$ for $F=\gamma_n^*\cS ol_{\AA^1_x}\big(\cD_{\AA^1_x}/\cD_{\AA^1_x}\mathrm{Hyp}(\alpha,\beta)\big)[1]$ is of the form 
\[S_\fb = 
\begin{pmatrix} 1 & {s_{12}} & {s_{13}} &  \dots & {s_{1n}} \\
     & 1 & {s_{23}} & \dots & {s_{2n}}  \\
     & & \ddots &  & \vdots \\
     & & & & 1 
\end{pmatrix}\] 
with $s_{ij}\in \{\lambda_1,\dots, \lambda_{n-1}\}$ and similarly for the lower triangular matrix $S_{-\fb}$ (the diagonal entries of the latter are then precisely $1-(1+\lambda_n)=-\lambda_n$). These two matrices relate the decompositions
$$\delta_{\pm\fa}\colon \pi^{-1}\CC_{H_{\pm\fa}}\otimes \cS ol^\mathrm{E}_{(\AA^1_w)_\infty}\big((eF)^\curlywedge\big)\overset{\sim}{\longrightarrow}\pi^{-1}\CC_{H_{\pm\fa}}\otimes \bigoplus_{\zeta\in\mu_n}\EE^{n\zeta w}$$
on the half-lines $h_{\fb}$ and $h_{-\fb}$. With the conventions as in §\ref{sec:StokesPhenomena} (in particular, using a counter-clockwise arrangement), Stokes matrices with respect to the generic direction $\frac{2\pi}{3n}$ are given by $S_{-\fb}^{-1}$ and $S_\fb$.

Pulling back along the map $g(v)=\frac{v^{n+m}}{n+m}$, we obtain the sectors $H_1,\ldots,H_{2(n+m)}$, alternately identified with $H_{-\fa}$ and $H_{\fa}$ via $g$, and decompositions
$$\pi^{-1}\CC_{S_j}\otimes \mathrm{E}\rho_n^{-1} \cS ol^\mathrm{E}_{(\AA^1_y)_\infty}(\cM_{n,m})\cong \pi^{-1}\CC_{S_j}\otimes \bigoplus_{\zeta\in\mu_n}\EE^{\frac{n}{n+m}\zeta v^{n+m}}$$
induced by $\delta_{-\fa}$ (resp.\ $\delta_\fa$) for $j$ odd (resp.\ even).
Consequently, the $2(n+m)$ Stokes matrices with respect to the generic direction $\theta_0=\frac{2\pi}{3n(n+m)}$ are given by $S_{2k}=S_\fb$ and $S_{2k-1}=S_{-\fb}^{-1}$ for $k=1,\dots,n+m$. 

To prove the statement about regularity note that if $\lambda_i \neq 0$ for $i=1,\dots,n-1$, we find that $\ker(1-S_\fb)$ is one-dimensional and thus $S_\fb$ is conjugate to a single Jordan block of maximal length. Therefore, the theorem will follow from Lemma \ref{lem:coeffs} below. 
\end{proof}
\begin{lemma}\label{lem:coeffs}
Let $r$ be a positive integer, $z\in \CC$ such that $z^i\neq z^j$, $z^i\neq 1$ for $i\neq j$, $1\le i,j \le r$ and let 
$g=\prod_{j=1}^r (X-z^j) \in \CC[X]$. 
Then no coefficient of $g$ vanishes. 
\end{lemma}
\begin{proof} Write $g=\sum_{i=0}^r a_iX^i$ and consider the matrix
\[M=\begin{pmatrix} 1 & z & z^2 & \dots & z^r \\
1 & z^2 & z^4 & \dots & z^{2r} \\
\vdots & & & & \vdots \\
1 & z^r & z^{2r} & \dots & z^{r^2} 
\end{pmatrix}.\] 
Let $v=(a_0,\dots,a_r)^T$ be the coefficient vector of $g$ and note that $Mv=0$. Removing any column of $M$ yields a square invertible matrix. Indeed up to scaling columns by powers of $z$ the resulting square matrix is a Vandermonde matrix with determinant $\prod_{j=1}^r (z^{k_j}-z^{l_j})$ where $k_j, l_j\in \{0,\dots,r\}$ and $k_j\neq l_j$ for all $j$. Since powers of $z$ are pairwise distinct, this determinant does not vanish. Assume there is an index $\ell$ such that $a_\ell=0$. Then denoting by $\widehat{M}$ the matrix we obtain from $M$ by removing the $(\ell+1)$-st column and by $v_\ell$ the vector obtained from $v$ by removing the entry $a_\ell$ we get that $\widehat{M}v_\ell=0$. This implies that every coefficient of $g$ vanishes, contradicting the fact that $g\neq 0$. 
\end{proof}

\begin{remark}
The proof of Theorem \ref{thm:stokesreg} gives a recipe to determine the Stokes matrices of a generalised Airy equation. Given any pair of coprime integers $(n,m)$ as an input, one can algorithmically compute the entries of the Stokes matrices of $\partial^n+y^m$ in terms of $(n+m)$-th roots of unity. Airy equations of type $(n,m)$ and $(m,n)$ are interchanged by Fourier--Laplace transform. The recipe makes it possible to compare Stokes matrices for both types. A thorough treatment of the behaviour of Stokes data under Fourier--Laplace transform can be found in \cite{Moc20}.
\end{remark}
\subsection{The rigid case $m=1$} Let us focus on the rigid operator
\[\partial_y^n-y.\]
This situation is a special case of Proposition 5.9 in \cite{BH89} and is recovered as follows. When $m=1$ the eigenvalues of $T_0$ and $T_\infty$ are 
\begin{align*}
    &\{1,\zeta_n,\zeta_n^2,\dots, \zeta_n^{n-1} \} \\
    &\{\zeta_{n+1}, \zeta_{n+1}^2, \dots, \zeta_{n+1}^n \}
\end{align*}
respectively. Then $T_1$ has order $2$ and is a real reflection. Let 
\[W=\{(x_1,\dots,x_{n+1}) \mid \sum x_i=0\}\] be the standard reflection representation of the symmetric group $S_{n+1}$. Define $\sigma_0\in \GL(W)$ to be the cyclic permutation $(1 \, \dots \,  n)$ and $\sigma_\infty$ to be the $(n+1)$-cycle $(n+1 \, \dots \, 1)$. Then $\sigma_1=\sigma_0^{-1}\sigma_\infty^{-1}$ is given by $(1, n+1)$. Hence $\sigma_0$ and $\sigma_\infty$ have the same eigenvalues as $T_0$ and $T_\infty$ and $\sigma_1$ is a reflection. By Levelt's Theorem \ref{thm:levelt} the monodromy representation $\rho$ of $\cH$ is isomorphic to the representation $$\pi_1(\AA^1_x\setminus \{0,1\},x_0)\rightarrow \GL(W)$$ defined by $\rho(\gamma_i)=\sigma_i$. We conclude that the local monodromy of $\gamma^*\cH$ at $n\mu_n$ is given by the collection $$\{(k, n+1), \, k=1,\dots,n\}.$$ 
In this case we can give a closed formula for the Stokes multipliers. In the notation of Theorem \ref{thm:stokesreg}, we have
\[\prod_{j=1}^n (X-\zeta_{n+1}^{-j})=X^n+X^{n+1}+\dots+X+1. \]
Therefore $s_{ij}=1$ for all $1\le i , j \le n$ and the Stokes multipliers of $\rho_n^*\cM_{n,1}$ are the $(n\times n)$-matrices given by 
\[S_{2k} =  \begin{pmatrix} 1 & 1 &   \dots & 1 \\
     & 1 &   \dots & 1  \\
     & & \ddots &   \vdots \\
     & &  & 1 
    \end{pmatrix} \\
     ,\quad S_{2k-1}^{-1} = \begin{pmatrix} -1 &   &  &  \\
    -1 & -1 &  &   \\
    \vdots  & \ddots & \ddots &   \\
    -1 &  \dots  &-1&  -1 
    \end{pmatrix} \]
for $k=1,\dots,n+1$, with respect to the generic direction $\theta_0=\frac{2\pi i}{3n(n+1)}$. 

\begin{remark}
Let us remark that Stokes matrices of the operator $\partial_y^n-y$ can also be computed analogously to the case of the classical Airy operator in \cite[§7.1]{DHMS20}, where the computation relies on the fact that $\cM_{n,1}=\cD_{\AA^1}/\cD_{\AA^1}(\partial_y^n-y)=\frF\big(\cE^{\frac{(-1)^n}{n+1}x^{n+1}}\big)$. This enables the authors in loc.\ cit.\ to reduce the computation to the study of the topology of a certain branched covering. However, for $m\neq 1$, we do not have such a nice realisation as a Fourier--Laplace transform of a single exponential, and these cases are still covered by our approach.
\end{remark}

\section{Differential Galois groups}
\subsection{An argument of Frenkel and Gross} Using the knowledge of Stokes matrices we recover results on the differential Galois group of the $\cD_{\AA^1_y}$-module $\cM_{n,m}$ due to Katz \cite[§4.2]{Ka87} and also Kamgarpour--Sage \cite{KS20}. Denote the differential Galois group by $H=G_{\diff}(\cM_{n,m})$.

\begin{prop} The differential Galois group of $\cM_{n,m}$ is given by
\[G_{\diff}(\cM_{n,m}) = \begin{cases} \SL(n), n \,\, \textup{odd} \\ \Sp(n), n \,\, \textup{even}. \end{cases} \]
\end{prop}
\begin{proof} This proof is basically borrowed from \cite[§13]{FG09}. Since $\cM_{n,m}$ is irreducible, the differential Galois group $H$ is connected reductive. It is easy to see that $\cM_{n,m}$ has trivial determinant, hence $H \subset \SL(n)$. \par
By \cite[Proposition 4.3]{Ka82} the differential Galois group $H'$ of $\rho_n^*\cM_{n,m}$ is contained in $H$. The Stokes matrices of $\rho_n^*\cM_{n,m}$ can be considered as elements of $H'$, cf.\ \cite[§8.3]{vdP03}. Since one of the Stokes matrices of $\rho_n^*\cM_{n,m}$ is regular unipotent, $H'$ and in turn $H$ contains a principal $\SL(2)$.  This puts a severe restriction on the possibilities for $H$, cf.\ \cite[§13]{FG09}. We can only have $H$ being any group appearing in the following chains
\begin{align*}
    \Sp(2r) &\subset \SL(2r), \\
    \SO(2r+1) &\subset \SL(2r+1),  \\
    G_2 \subset \SO(7) &\subset \SL(7).
\end{align*}
The differential Galois group $H$ contains the formal monodromy $w$ and the exponential torus $S$. By Proposition \ref{prop:hyptoairy}  the exponential factors of $\rho_n^*\cM_{n,m}$ are $\zeta_n^jv^{n+m}$ (up to a scalar) for $j\in\{0,\ldots,n-1\}$. This implies that the centraliser $T$ of $S$ is a maximal torus.
Since $w$ normalises $S$ it normalises $T$. Again by Proposition \ref{prop:hyptoairy} the formal monodromy $w$ cyclically permutes the exponential factors (in a suitable ordering), hence its image in $N(T)/T=W$ is a Coxeter element of order $n$. Therefore we can exclude all groups above which have smaller Coxeter number. For odd $n$ this proves that $H=\SL(n)$. \par
Finally it is easy to see that for even $n$ the operator $P_{n,m}=\partial_y^n - y^m$ is self-dual. This implies that $G$ leaves a non-degenerate bilinear form invariant. By what we said above, this form has to be alternating and $H=\Sp(n)$. 
\end{proof}

\begin{remark} Let $G$ be a simple complex algebraic group with Lie algebra $\frg$. Fix a maximal torus $T$ and a Borel subgroup $B$ of $G$. For any negative simple root $\alpha$ let $x_\alpha$ be a generator of the root space $\frg_\alpha$, $N=\sum_{\alpha} x_{\alpha}$ and $E$ a generator of the highest root space. Then in \cite{KS20} a generalisation of the Airy equation is defined as the connection
\[\nabla = d+(N+tE)dt\]
on the trivial vector bundle on $\AA^1$. Based on our computations for $\GL(n)$ one may expect that the Stokes matrices of $\nabla$ are regular unipotent. The differential Galois groups could then be determined in a similar way. Unfortunately our method does not generalise beyond type $A$ as it relies heavily on Fourier transform. 
\end{remark}

\newpage
\bibliographystyle{alpha}
\bibliography{stokesbib}

\end{document}